\documentclass{amsart}

\newcommand{\F}{\mathcal F}
\newcommand{\K}{\mathcal K}
\newcommand{\IR}{\mathbb R}
\newcommand{\IN}{\mathbb N}

\newcommand{\w}{\omega}
\newcommand{\osc}{\mathrm{osc}}
\newcommand{\dec}{\mathrm{dec}}
\newcommand{\ddec}{\mathrm{ddec}}
\newcommand{\Dec}{\mathrm{dec}}

\newtheorem{theorem}{Theorem}
\newtheorem{proposition}{Proposition}
\newtheorem{corollary}{Corollary}
\newtheorem{problem}{Problem}

\title{On $\sigma$-convex subsets in spaces of scatteredly continuous functions}%
\author{Taras Banakh, Bogdan Bokalo, Nadiya Kolos}
\address{Ivan Franko National University of Lviv}%
\email{\small tbanakh@yahoo.com, bogdanbokalo@mail.ru,
nadiya\_kolos@ukr.net}

\begin{document}
\begin{abstract}{We prove that for any topological space $X$ of countable tightness,
each $\sigma$-convex subspace $\F$ of the space $SC_p(X)$ of
scatteredly continuous real-valued functions on $X$ has network
weight $nw(\F)\le nw(X)$. This implies that for a metrizable
separable space $X$, each compact convex subset in the function
space $SC_p(X)$ is metrizable. Another corollary says that two
Tychonoff spaces $X,Y$ with countable tightness and topologically
isomorphic linear topological spaces $SC_p(X)$ and $SC_p(Y)$ have
the same network weight $nw(X)=nw(Y)$. Also we prove that each zero-dimensional separable Rosenthal compact space is homeomorphic to a compact subset of the function space $SC_p(\w^\w)$ over the space $\w^\w$ of irrationals.}
\end{abstract}
\maketitle

%\section*{Introduction}

This paper was motivated by the problem of studying the
linear-to\-po\-lo\-gi\-cal structure of the space  $SC_p(X)$ of
scatteredly continuous real-valued functions on a topological space
$X$, addressed in \cite{BK1,BK2}.

A function $f:X\to Y$ between two topological spaces is called {\em
scatteredly continuous} if for each non-empty subspace $A\subset X$
the restriction $f|A:A\to Y$ has a point of continuity. Scatteredly
continuous functions were introduced in \cite{AB} (as almost
continuous functions) and studied in details in \cite{BM}, \cite{BB}
and \cite{BK3}. If a topological space $Y$ is regular, then the
scattered continuity of a function $f:X\to Y$ is equivalent to the
weak discontinuity of $f$; see \cite{AB}, \cite[4.4]{BB}. We recall
that a function $f:X\to Y$ is {\em weakly discontinuous} if each
subspace $A\subset X$ contains an open dense subspace $U\subset A$
such that the restriction $f|U:U\to Y$ is continuous.

For a topological space $X$ by $SC_p(X)\subset\IR^X$ we denote the
linear space of all scatteredly continuous (equivalently, weakly
discontinuous) functions on $X$, endowed with the topology of
pointwise convergence. It is clear that the space $SC_p(X)$ contains
the linear subspace $C_p(X)$ of all continuous real-valued functions
on $X$. Topological properties of the function spaces $C_p(X)$ were
intensively studied by topologists, see \cite{Arh}. In particular,
they studied the interplay between topological invariants of
topological space $X$ and its function space $C_p(X)$.

Let us recall \cite{En,Juh} that for a topological space $X$ its
\begin{itemize}
\item {\em weight} $w(X)$ is the smallest cardinality of a base of the topology of $X$;
\item {\em network weight} $w(X)$ is the smallest cardinality of a network of the topology of $X$;
\item {\em tightness} $t(X)$ is the smallest infinite cardinal $\kappa$ such that for each subset $A\subset X$ and a point $a\in \bar A$ in its closure there is a subset $B\subset A$ of cardinality $|B|\le\kappa$ such that $a\in \bar B$;
\item {\em Lindel\"of number} $l(X)$ is the smallest infinite cardinal $\kappa$ such that each open cover of $X$ has a subcover of cardinality $\le\kappa$;
\item {\em hereditary Lindel\"of number} $hl(X)=\sup\{l(Z):Z\subset X\}$;
\item {\em density $d(X)$} if the smallest cardinality of a dense subset of $X$;
\item {\em the hereditary density} $hd(X)=\sup\{d(Z):Z\subset X\}$;
\item {\em spread} $s(X)=\sup\{|D|:D$ is a discrete subspace of $X\}$.
\end{itemize}

By \cite[\S I.1]{Arh}, for each Tychonoff space $X$ the function
space $C_p(X)$ has weight $w(C_p(X))=|X|$ and network weight
$nw(SC_p(X))=nw(X)$. For the function space $SC_p(X)$ the situation
is a bit different.

\begin{proposition} For any $T_1$-space $X$ we have $$s\big(SC_p(X)\big)=nw\big(SC_p(X)\big)=w\big(SC_p(X)\big)=|X|.$$
\end{proposition}

\begin{proof} It is clear that $s\big(SC_p(X)\big)\le nw\big(SC_p(X)\big)\le w\big(SC_p(X)\big)\le w(\IR^X)=|X|$.
To  see that $|X|\le s\big(SC_p(X)\big)$, observe that for each point $a\in X$ the characteristic function
$$\delta_a:X\to\IR=\begin{cases}1,&\mbox{if $x=a$}\\
0,&\mbox{otherwise}
\end{cases}
$$
of the singleton $\{a\}$ is scatteredly continuous, and the subspace
$\mathcal D=\{\delta_a:a\in X\}\subset SC_p(X)$ has cardinality
$|X|$ and is discrete in $SC_p(X)$.
\end{proof}

The deviation of a subset $\F\subset SC_p(X)$ from being a subset of
$C_p(X)$ can be measured with help of the cardinal number $\Dec(\F)$
called the {\em decomposition number} of $\F$. It is defined as the
smallest cardinality $|\mathcal{C}|$ of a cover $\mathcal C$ of $X$
such that for each $C\in\mathcal{C}$ and $f\in\F$ the restriction
$f|C$ is continuous. If the function family $\F$ consists of a
single function $f$, then the decomposition number
$\Dec(\F)=\Dec(\{f\})$ coincides with the decomposition number
$\dec(f)$ of the function $f$, studied in \cite{Sol}. It is clear
that $\Dec\big(C_p(X)\big)=1$.

\begin{proposition} For a $T_1$ topological space $X$ the decomposition
number $\Dec(SC_p(X))$ is equal to the decomposition number
$\Dec(\mathcal D)$ of the subset $\mathcal D=\{\delta_a:a\in X\}\subset
 SC_p(X)$ and is equal to the smallest cardinality $\ddec(X)$ of a
 cover of $X$ by discrete subspaces.
\end{proposition}

\begin{proof} It is clear that $\Dec(\mathcal D)\le\Dec(SC_p(X))\le \ddec(X)$.
To prove that $\Dec(\mathcal D)\ge\ddec(X)$, take a cover
$\mathcal{C}$ of $X$ of cardinality $|\mathcal{C}|=\Dec(\mathcal D)$
such that for each $C\in\mathcal{C}$ and each characteristic
function $\delta_a\in\mathcal D$ the restriction $\delta_a|C$ is
continuous. We claim that each space $C\in\mathcal{C}$ is discrete.
Assuming conversely that $C$ contains a non-isolated point $c\in C$,
observe that for the characteristic function $\delta_c$ of the
singleton $\{c\}$ the restriction $\delta_c|C$ is not continuous.
But this contradicts the choice of the cover $\mathcal{C}$.
Therefore the cover $\mathcal{C}$ consists of discrete subspaces of
$X$ and $\ddec(X)\le|\mathcal{C}|=\Dec(\mathcal D)$.
\end{proof}

In contrast to the whole function space $SC_p(X)$ which has large
decomposition number $\Dec(SC_p(X))$, its $\sigma$-convex subsets
have decomposition numbers bounded from above by the hereditary
Lindel\"of number of $X$.

Following \cite{Al} and \cite{TUZ}, we define a subset $C$ of a
linear topological space $L$ to be  {\em $\sigma$-convex} if
for any sequence of points $(x_n)_{n\in\w}$ in $C$ and any sequence
of positive real numbers $(t_n)_{n\in\w}$ with
$\displaystyle\sum_{n=0}^\infty t_n=1$ the series
$\displaystyle\sum_{n=0}^\infty t_nx_n$ converges to some point
$c\in C$. It is easy to see that each compact convex subset
$K\subset L$ is $\sigma$-convex. On the other hand, each
$\sigma$-convex subset of a linear topological space $L$ is
necessarily convex and bounded in $L$.

The main result of this paper is the following:
\begin{theorem}\label{main} For any topological space $X$ of countable
tightness, each $\sigma$-convex subset $\F\subset SC_p(X)$ has decomposition number $\Dec(\F)\le hl(X)$.
\end{theorem}

This theorem will be proved in Section~\ref{s:pf-t2}. Now we derive
some simple corollaries of this theorem.

\begin{corollary}\label{c1} For any topological space $X$ of countable
tightness, each $\sigma$-convex subset $\F\subset SC_p(X)$ has network
weight $nw(\F)\le nw(X)$. Moreover, $$nw(X)=\max\{nw(\F):\mbox{$\F$ is
a $\sigma$-convex subset of $SC_p(X)$}\}$$provided the space $X$ is Tychonoff.
\end{corollary}

\begin{proof} By Theorem~\ref{main}, each $\sigma$-convex subset
$\F\subset SC_p(X)$ has decomposition number $\Dec(\F)\le hl(X)$.
Consequently, we can find a disjoint cover $\mathcal{C}$ of $X$ of
cardinality $|\mathcal{C}|=\Dec(\F)\le hl(X)$ such that for each
$C\in\mathcal{C}$ and $f\in\F$ the restriction $f|C$ is continuous.

Let $Z=\oplus \mathcal{C}=\{(x,C)\in X\times\mathcal{C}:x\in
C\}\subset X\times \mathcal{C}$ be the topological sum of the family
$\mathcal{C}$, and $\pi:Z\to X$, $\pi:(x,C)\mapsto x$, be the
natural projection of $Z$ onto $X$. Since the cover $\mathcal{C}$ is
disjoint, the map $\pi:Z\to X$ is bijective and hence induces a
topological isomorphism $\pi^*:\IR^X\to \IR^Z$, $\pi^*:f\mapsto
f\circ \pi$. The choice of the cover $\mathcal{C}$ guarantees that
$\pi^*(\F)\subset C_p(Z)$. By (the proof of) Theorem~I.1.3 of
\cite{Arh}, $nw(C_p(Z))\le nw(Z)$ and hence
$$
\begin{aligned}
nw(\F)&=nw(\pi^*(\F))\le nw(C_p(Z))\le nw(Z)\le\\
&\le |\mathcal{C}|\cdot nw(X)\le
hl(X)\cdot nw(X)=nw(X).
\end{aligned}$$

If the space $X$ is Tychonoff, then the ``closed unit ball''
$$\mathcal B=\{f\in C_p(X):\displaystyle\sup_{x\in X}|f(x)|\le
1\}\subset C_p(X)$$ is $\sigma$-convex and has network weight
$nw(\mathcal B)=nw(X)$ according to Theorem I.1.3 of \cite{Arh}. So,
$$nw(X)=\max\{nw(\F):\mbox{$\F$ is a $\sigma$-convex subset of
$SC_p(X)$}\}.$$
\end{proof}

In the same way we can derive some bounds on the weight of compact
convex subsets in function spaces $SC_p(X)$.

\begin{corollary}\label{c2} For any topological space $X$ of countable
tightness, each compact convex subset $\K\subset SC_p(X)$ has
weight $w(\K)\le \max\{hl(X),hd(X)\}$. Moreover,
$$
\begin{aligned}
hl(X)\le \sup\{w(\K):&\mbox{ $\K$ is a compact convex subset of } SC_p(X)\}\le\\
&\le\max\{hl(X),hd(X)\}.
\end{aligned}
$$
\end{corollary}

\begin{proof} Given a compact convex subset $\K\subset SC_p(X)$,
use Theorem~\ref{main} to find a disjoint cover $\mathcal{C}$ of $X$
of cardinality $|\mathcal{C}|=\Dec(\K)\le hl(X)$  such that for each
$C\in\mathcal{C}$ and $f\in\K$ the restriction $f|C$ is continuous.
Let $Z=\oplus \mathcal{C}$ and $\pi:\oplus \mathcal{C}\to X$ be the
natural projection, which induces a linear topological isomorphism
$\pi^*:\IR^X\to \IR^Z$, $\pi^*:f\mapsto f\circ\pi$, with
$\pi^*(\K)\subset C_p(Z)$. It follows that the topological sum
$Z=\oplus\mathcal{C}$ has density
$d(Z)\le\displaystyle\sum_{C\in\mathcal{C}}d(C)\le|\mathcal{C}|\cdot
hd(X)\le\max\{hl(X),hd(X)\}$, and so we can fix a dense subset
$D\subset Z$ of cardinality $|D|=d(Z)\le\max\{hl(X),hd(X)\}$. Since
the restriction operator $R:C_p(Z)\to C_p(D)$, $R:f\mapsto f|D$, is
injective and continuous, we conclude that
$$
\begin{aligned}
w(\K)&=w(\pi^*(\K))=w(R\circ \pi^*(\K))\le w
(\IR^D)=\\
&=|D|\cdot\aleph_0\le\max\{hl(X),hd(X)\}.
\end{aligned}
$$

Next, we show that $hl(X)\le\tau$ where $$\tau=\sup\{w(\K):\mbox{$\K$ is a compact convex subset of
$SC_p(X)$}\}.$$ Assuming conversely that $hl(X)>\tau$ and using the equality $hl(X)=\sup\{|Z|:Z\subset X$ is scattered$\}$ established in \cite{Juh}, we can find a scattered subspace $Z\subset X$ of cardinality $|Z|>\tau$. It is easy to check that
each function $f:X\to[0,1]$ with $f(X\setminus Z)\subset\{0\}$ is
scatteredly continuous, which implies that the subset
$$\K_Z=\big\{f\in SC_p(X):f(Z)\subset[0,1],\;f(X\setminus
Z)\subset\{0\}\big\}$$ is compact, convex and homeomorphic to the
Tychonoff cube $[0,1]^Z$. Then $\tau\ge w(\K_Z)=w([0,1]^Z)=|Z|>\tau$ and this is a desired contradiction that completes the proof.\end{proof}

Corollaries~\ref{c1} or \ref{c2} imply:

\begin{corollary}\label{c3} For a metrizable separable space $X$, each compact convex subspace $\K\subset SC_p(X)$ is metrizable.
\end{corollary}

Finally, let us observe that Corollary~\ref{c1} implies:

\begin{corollary}\label{c4}  If for Tychonoff spaces $X,Y$ with
countable tightness the linear topological spaces $SC_p(X)$ and
$SC_p(Y)$ are topologically isomorphic, then $nw(X)=nw(Y)$.
\end{corollary}

\section{Weakly discontinuous families of functions}

In this section we shall generalize the notions of scattered
continuity and weak discontinuity to function families.

A family of
functions $\F\subset Y^X$ from a topological space $X$ to a
topological space $Y$ is called
\begin{itemize}
\item {\em scatteredly continuous} if each non-empty subset
$A\subset X$ contains a point $a\in A$ at which each function
$f|A:A\to Y$, $f\in\F$ is continuous;
\item {\em weakly discontinuous}  if each subset $A\subset X$ contains
an open dense subspace $U\subset A$ such that each function
$f|U:U\to Y$, $f\in\F$ is continuous.
\end{itemize}

The following simple characterization can be derived from the
corresponding definitions and Theorem~4.4 of \cite{BB} (saying that
each scatteredly continuous function with values in a regular
topological space is weakly discontinuous).

\begin{proposition}\label{p3} A function family $\F\subset Y^X$
is scatteredly continuous (resp. weakly discontinuous) if and only
if so is the function $\Delta\F:X\to Y^\F$, $\Delta\F:x\mapsto(f(x))_{f\in\F}$.
Consequently, for a regular topological space $Y$, a function family
$\F\subset Y^X$ is scatteredly continuous if and only if it is weakly discontinuous.
\end{proposition}

Propositions 4.7 and 4.8 \cite{BB} imply that each weakly discontinuous function $f:X\to Y$ has decomposition number $\dec(f)\le hl(X)$. This fact combined with Proposition~\ref{p3} yields:

\begin{corollary}\label{c5} For any topological spaces $X,Y$,
each weakly discontinuous function family $\F\subset Y^X$ has
decomposition number $\Dec(\F)\le hl(X)$.
\end{corollary}

\section{Weak discontinuity of $\sigma$-convex sets in function spaces}

For a topological space $X$ by $SC_p^*(X)$ we denote the space of
all {\em bounded} scatteredly continuous real-valued functions on
$X$. It is a subspace of the function space $SC_p(X)\subset \IR^X$.
Each function $f\in SC_p^*(X)$ has finite norm
$\|f\|=\displaystyle\sup_{x\in X}|f(x)|$.

\begin{theorem}\label{wd} For any topological space $X$ with
countable tightness, each $\sigma$-convex subset $\F\subset SC^*_p(X)$ is weakly discontinuous.
\end{theorem}

\begin{proof} By Proposition~\ref{p3}, the weak discontinuity of
the function family $\F$ is equivalent to the scattered continuity
of the function $\Delta\F:X\to\IR^\F$, $\Delta\F:x\mapsto(f(x))_{f\in\F}$.
Since the space $X$ has countable tightness, the scattered
continuity of $\Delta\F$ will follow from Proposition~2.3 of
\cite{BB} as soon as we check that  for each countable subset
$Q=\{x_n\}_{n=1}^\infty\subset X$ the restriction
$\Delta\F|Q:Q\to\IR^\F$ has a continuity point. Assuming the
converse, for each point $x_n\in Q$ we can choose a function $f_n\in
\F$ such that the restriction $f_n|Q$ is discontinuous at $x_n$.

Observe that a function $f:Q\to\IR$ is discontinuous at a point
$q\in Q$ if and only if it has strictly positive oscillation
$$\osc_q(f)=\inf_{O_q}\sup\{|f(x)-f(y)|:x,y\in O_q\}$$at the point $q$.
In this definition the infimum is taken over all neighborhoods $O_q$
of $q$ in $Q$.

We shall inductively construct a sequence $(t_n)_{n=1}^\infty$ of
positive real numbers such that for every $n\in\IN$ the following
conditions are satisfied:
\begin{enumerate}
\item[1)] $t_1\le \frac12$, $t_{n+1}\le \frac12t_n$, and $t_{n+1}\cdot\|f_{n+1}\|\le \frac12t_n\cdot\|f_n\|$,
\item[2)] the function $s_n=\displaystyle\sum_{k=1}^nt_kf_k$ restricted to $Q$ is discontinuous at $x_n$,
\item[3)] $t_{n+1}\cdot\|f_{n+1}\|\le\frac18\osc_{x_n}(s_n|Q)$.
\end{enumerate}

We start the inductive construction letting $t_1=1/2$. Then the
function $s_1|Q=t_1\cdot f_1|Q$ is discontinuous at $x_1$ by the
choice of the function $f_1$. Now assume that for some $n\in\IN$
positive numbers $t_1\dots,t_n$ has been chosen so that the function
$s_n=\displaystyle\sum_{k=1}^nt_kf_k$ restricted to $Q$ is
discontinuous at $x_n$.

Choose any positive number $\tilde t_{n+1}$ such that
$$\tilde t_{n+1}\le \frac12t_n,\;\;\tilde t_{n+1}\cdot\|f_{n+1}\|\le\tfrac
12t_n\cdot\|f_{n}\|\mbox{ \ and \ }  \tilde
t_{n+1}\cdot\|f_{n+1}\|\le\tfrac18\osc_{x_n}(s_n|Q),$$ and consider
the function $\tilde s_{n+1}=s_n+\tilde t_{n+1}f_{n+1}$. If the
restriction of this function to $Q$ is discontinuous at the point
$x_{n+1}$, then put $t_{n+1}=\tilde t_{n+1}$ and finish the
inductive step. If $\tilde s_{n+1}|Q$ is continuous at $x_{n+1}$,
then put $t_{n+1}=\frac12\tilde t_{n+1}$ and observe that the
restriction of the function
$$s_{n+1}=\displaystyle\sum_{k=1}^{n+1} t_kf_k=s_n+\tfrac12\tilde t_{n+1}f_{n+1}=
\tilde s_{n+1}-\tfrac12\tilde t_{n+1}f_{n+1}$$ to $Q$ is discontinuous
at $x_{n+1}$. This completes the inductive construction.
\smallskip

The condition (1) guarantees that  $\displaystyle\sum_{n=1}^\infty
t_n\le  1$ and hence the number
$t_0=1-\displaystyle\sum_{n=1}^\infty t_n$ is non-negative. Now take
any function $f_0\in\F$ and consider the function
$$s=\displaystyle\sum_{n=0}^\infty t_nf_n$$ which is well-defined and belongs to
$\F$ by the $\sigma$-convexity of $\F$.

The functions $f_0,s\in\F\subset SC_p(X)$ are weakly discontinuous
and hence for some open dense subset $U\subset Q$ the restrictions
$s|U$ and $f_0|U$ are continuous. Pick any point $x_n\in U$. Observe
that $$s=t_0f_0+s_n+\displaystyle\sum_{k=n+1}^\infty t_{k}f_k$$and
hence
$$s_n=s-t_0f_0-\displaystyle\sum_{k=n+1}^\infty t_kf_k=s-t_0f_0-u_n,$$
where $u_n=\displaystyle\sum_{k=n+1}^\infty t_kf_k$. The conditions
(1) and (3) of the inductive construction guarantee that the
function $u_n$
 has norm
$$\|u_n\|\le\displaystyle\sum_{k=n+1}^\infty t_k\|f_k\|\le 2 t_{n+1}\|f_{n+1}\|\le
\frac14\osc_{x_n}(s_n|Q).$$
Since $s_n=s-t_0f_0-u_n$, the triangle inequality implies that
$$0<\osc_{x_n}(s_n|Q)\le \osc_{x_n}(s|Q)+\osc_{x_n}(t_0f_0|Q)+
\osc_{x_n}(u_n)\le$$ $$\le 0+0+2\|u_n\|\le\frac12\osc_{x_n}(s_n|Q)$$
which is a desired contradiction, which shows that the restriction
$\Delta\F|Q$ has a point of continuity and the family $\F$ is weakly
discontinuous.
\end{proof}

\section{Proof of Theorem~\ref{main}}\label{s:pf-t2}

Let $X$ be a topological space with countable tightness and $\F$ be
a $\sigma$-convex subset in the function space $SC_p(X)$. The
$\sigma$-convexity of $\F$ implies that for each point $x\in X$ the
subset $\{f(x):f\in\F\}\subset\IR$ is bounded (in the opposite case
we could find sequences $(f_n)_{n\in\w}\in\F^\w$ and
$(t_n)_{n\in\w}\in[0,1]^\w$ with $\displaystyle\sum_{n=0}^\infty
t_n=1$ such that the series $\displaystyle\sum_{n=1}^\infty
t_nf_n(x)$ is divergent). Then $X=\displaystyle\bigcup_{n=1}^\infty
X_n$ where $X_n=\{x\in X:n\le \displaystyle\sup_{f\in\F}|f(x)|<n+1\}$  for $n\in\w$.

It follows that for every $n\in\w$ the family $\F|X_n=\{f|X_n:f\in
\F\}$ is a $\sigma$-convex subset of the function space
$SC_p^*(X_n)$. By Theorem~\ref{wd}, the function family $\F|X_n$ is
weakly discontinuous and by Corollary~\ref{c5}, $\Dec(\F|X_n)\le
hl(X_n)$. Then $\Dec(\F)\le\displaystyle\sum_{n=0}^\infty
\Dec(\F|X_n)\le\displaystyle\sum_{n=0}^\infty hl(X_n)\le hl(X)$.

\section{Some Open Problems}

The presence of the condition of countable tightness in
Theorem~\ref{main} and its corollaries suggests the following open
problem.

\begin{problem} Is it true $w(\K)\le nw(X)$ for each topological
space $X$ and each compact convex subset $\K\subset SC_p(X)$?
\end{problem}

By Theorem~\ref{wd}, for each topological space $X$ of countable tightness, each compact convex subset $\K\subset SC_p^*(X)$ is weakly discontinuous.

\begin{problem} For which topological spaces $X$ each compact
convex subset $\K\subset SC_p(X)$ is weakly discontinuous?
\end{problem}

According to Corollary~\ref{c3}, each compact convex subset $\K\subset SC_p(\w^\w)$ is metrizable.

\begin{problem} Is a compact subset $\K\subset SC_p(\w^\w)$ metrizable if $K$ is  homeomorphic to a compact convex subset of $\IR^{\mathfrak c}$.
\end{problem}

Let us recall that a topological space $K$ is {\em Rosenthal compact} if $K$ is homeomorphic to a compact subspace of the space $\mathcal B_1(X)\subset\IR^X$ of functions of the first Baire class on a Polish space $X$. In this definition the space $X$ can be assumed to be equal to the space $\w^\w$ of irrationals.

\begin{problem}\label{pr4} Is each Rosenthal compact space homeomorphic to a compact subset of the function space $SC_p(\w^\w)$?
\end{problem}

This problem has affirmative solution in the realm of zero-dimensional separable Rosenthal compacta.

\begin{theorem} Each zero-dimensional separable Rosenthal compact space $K$ is homeomorphic to a compact subset of the function space $SC_p(\w^\w)$.
\end{theorem}

\begin{proof} Let $D\subset K$ be a countable dense subset in $K$. Let $A=C_D(K,2)$ be the space of continuous functions $f:K\to 2=\{0,1\}$ endowed with the smallest topology making the restriction operator $R:C_D(K,2)\to 2^D$, $R:f\mapsto f|D$, continuous. By the characterization of separable Rosenthal compacta \cite{God}, the space $A$ is analytic, i.e., $A$ is the image of the Polish space $X=\w^\w$  under a continuous map $\pi:X\to A$. Now consider the map $\delta:K\to 2^A$, $\delta:x\mapsto (f(x))_{f\in A}$. This map is continuous and injective by the zero-dimensionality of $K$. The map $\pi:X\to A$ induces a homeomorphism $\pi^*:2^A\to 2^{X}$, $\pi^*:f\mapsto f\circ\pi$. Then $\pi^*\circ\delta:K\to 2^{X}$ is a topological embedding.

We claim that $\pi^*\circ\delta(K)\subset SC_p(X)\cap 2^X$. Given a point $x\in K$, we need to check that the function $\pi^*\circ \delta(x)\in 2^X$ is scatteredly continuous. It will be convenient to denote the function $\delta(x)\in 2^A$ by $\delta_x$. This function assigns to each $f\in A=C_D(K)$ the number $\delta_x(f)=f(x)\in 2$.

By \cite{Ros,BFT}, the Rosenthal compact space $K$ is
Fr\'echet-Urysohn, so there is a sequence $(x_n)_{n\in\w}\in D^\w$
with $\displaystyle\lim_{n\to\infty}x_n=x$. Then the function
$\delta_x:A\to 2$, $\delta_x:f\mapsto f(x)$, is the pointwise limit
of the continuous functions $\delta_{x_n}$, which implies that
$\delta_x$ is a function of the first Baire class on $A$ and
$\delta_x\circ \pi:X\to 2$ is a function of the first Baire class on
the Polish space $X$. Since this function has discrete range, it is
scatteredly continuous by Theorem 8.1 of \cite{BB}. Consequently,
$\pi^*\circ\delta(x)\in SC_p(X)$ and $K$ is homeomorphic to the
compact subset $\pi^*\circ\delta(K)\subset SC_p(X)$.
\end{proof}

A particularly interesting instance of Problem~\ref{pr4} concerns non-metrizable convex Rosenthal compacta. One of the simples spaces of this sort is the Helly space. We recall that the {\em Helly space} is the subspace of $B_1(I)$ consisting of all non-decreasing functions $f:I\to I$ of the unit interval $I=[0,1]$.

\begin{problem} Is the Helly space homeomorphic to a compact subset of the function space $SC_p(\w^\w)$?
\end{problem}

%\newpage

\end{document}